\documentclass{article}
\usepackage{amsmath, amssymb, amsthm}
\usepackage{hyperref}
\usepackage{ellipsis}
\usepackage{enumerate}
\usepackage{hyperref}
\numberwithin{equation}{section}

\newcommand{\nc}{\newcommand}
\nc{\Qq}{\mathbb{Q}}
\nc{\Cc}{\mathbb{C}}
\nc{\Rr}{\mathbb{R}}
\nc{\Zz}{\mathbb{Z}}
\nc{\Nn}{\mathbb{N}}
\nc{\Pp}{\mathbb{P}}

\nc{\cA}{\mathcal{A}}
\nc{\cP}{\mathcal{P}}
\nc{\cC}{\mathcal{C}}
\nc{\cB}{\mathcal{B}}
\nc{\cE}{\mathcal{E}}
\nc{\cF}{\mathcal{F}}
\nc{\cS}{\mathcal{S}}
\nc{\cH}{\mathcal{H}}
\nc{\cZ}{\mathcal{Z}}
\nc{\cM}{\mathcal{M}}
\nc{\cN}{\mathcal{N}}
\nc{\cI}{\mathcal{I}}
\nc{\cT}{\mathcal{T}}

\nc{\mf}{\mathfrak{m}}
\nc{\ba}{ {\boldsymbol \alpha}}

\nc{\norm}[1]{\| #1 \|} \nc{\gen}[1]{\langle #1 \rangle} 

\nc{\wL}[3]{\det
  \begin{pmatrix} #1 & \cdots & #2 \\
  \partial_L #1 & \cdots & \partial_L #2 \\ \vdots & \phantom{\vdots} &
  \vdots \\ \partial_L^{#3} #1 & \cdots & \partial_L^{#3} #2
\end{pmatrix}}

\nc{\dO}{\partial_{\Omega}}
\nc{\id}{\mathrm{id}}

\nc{\DMO}{\DeclareMathOperator}
\renewcommand{\vec}{\mathbf}

\DMO{\rk}{rank} 
\DMO{\Der}{Der} 
\DMO{\td}{td} 
\DMO{\Log}{Log} 
\DMO{\Exp}{Exp}
\DMO{\supp}{supp}
\DMO{\ord}{ord}
\DMO{\sgn}{sgn}
\DMO{\diag}{diag}

\newtheorem*{thm*}{Jacobian Criterion}
\newtheorem{thm}{Theorem}[section]
\newtheorem{lem}[thm]{Lemma}
\newtheorem{prop}[thm]{Proposition}
\newtheorem{cor}[thm]{Corollary}

\theoremstyle{remark}
\newtheorem{exa}{Example}[section]

\title{Applications of differential algebra to algebraic independence of
arithmetic functions} \author{Wai Yan Pong \\ California State University
Dominguez-Hills}

\begin{document}
\maketitle 
\abstract{We generalize and unify the proofs of several results on algebraic
independence of arithmetic functions and Dirichlet series by a theorem of Ax
on differential Schanuel conjecture. Along the way, we find counter-examples
to some results in the literature.}
 
\section{Introduction} 
\label{sec:Axthm}

Schanuel Conjecture asserts that for any $\Qq$-linearly independent complex
numbers $a_1,\dots, a_n$ there are at least $n$ numbers among
  \begin{align*}
    a_1, \dots,a_n, \exp(a_1),\dots, \exp(a_n)
  \end{align*}
that are algebraically independent over the rational numbers. It is well-known
that a number of remarkable results about transcendental numbers:
Lindemann-Weierstrass Theorem, Gelfond-Schneider Theorem and Baker's Theorem to
name a few are consequences of this statement. For the state of the art on this
topic, we refer the reader to Waldschmidt's paper~\cite{miw}. In this article,
we argue that Schanuel's insight remains valid for arithmetic functions.  We
improve several existing results on algebraic independence of arithmetic
functions by applying an analog of Schanuel Conjecture for differential rings.
More precisely, we deduce them from the following theorem of James
Ax~\cite[Theorem~3]{ax}:
\begin{thm}
  \label{th:genax} Let $F/C/\Qq$ be a tower of fields. Suppose $\Delta$
  is a set of derivations of $F$ with $\bigcap_{D \in \Delta} \ker D =
  C$. Let $y_1,\dots, y_n$, $z_1,\dots, z_n \in F^{\times}$ be such that
  \begin{enumerate}
    \item[\upshape{(a)}] \label{i:ex-log} for all $D \in \Delta$, $i =
    1,\dots,n$, $Dy_i= Dz_i/z_i$ and either
    \item[\upshape{(b)}] \label{i:nontrivial} no non-trivial power product
  of the $z_j$ is in $C$, or
    \item[\upshape{(b$^\prime$)}] \label{i:modulo} the $y_i$ are $\Qq$-linearly
    independent modulo $C$.
  \end{enumerate}
  Then
  \begin{align*}
    \td_C C(y_1,\dots, y_n, z_1,\dots, z_n) \ge n + \rk\left( Dy_i
    \right)_{\begin{smallmatrix} & D \in \Delta \\ &1 \le i \le n
    \end{smallmatrix}}.
  \end{align*}
\end{thm}
A word about terminology. Let $G$ be an abelian group (written
multiplicatively). We say that $g_1,\ldots, g_n \in G$ are (or the family
$g_1,\ldots, g_n$ is) {\bf multiplicatively independent} if the equation
$g_1^{k_1}\cdots g_n^{k_n}=1$ implies the integers $k_1,\ldots, k_n$ are
all zeros. The implication is vacuously true for the empty family and so
it is multiplicatively independent. A subset $X$ of $G$ is multiplicatively
independent if every finite family of $X$ is. For $H$ a subgroup of $G$, we
say that $X$ is {\bf multiplicatively independent modulo $H$}, if the image
of $X$ in the quotient group $G/H$ is multiplicatively independent. We will
use these terminologies throughout and first, let us restate Condition~(b)
in Theorem~\ref{th:genax} as ``the $z_i$ are multiplicatively independent
modulo $C^{\times}$''. We prefer doing so because that draws a closer analogy
between Condition~(b) and~(b').

\section{Arithmetic Functions} \label{sec:arithfun} 

In this section we introduce the notations and summarize the facts about
arithmetic functions that we will use subsequently. The reader can
consult~\cite[Chapter~2]{ant} and~\cite[Chapter~4]{hint} for more information.
We use $\Pp$ to denote the set of primes and $p$ will always stand for a prime
in this article.

{\bf Arithmetic functions} are complex-valued functions with domain the set of
natural numbers. It is beneficial at times to think of them as functions on
$\Rr$ vanishing at points that are not natural numbers. Arithmetic functions
form a commutative ring $\cA$ under pointwise addition of functions $+$
and convolution product $*$ defined as:
\begin{align*}
  (f * g)(n) = \sum_{d|n} f(d)g\left( \frac{n}{d} \right).
\end{align*}
Identifying $\alpha \in \Cc$ with the function $1 \mapsto \alpha, n \mapsto 0\
(n >1)$ turns $\cA$ into a $\Cc$-algebra. Under this identification $0$ and $1$
become the neutral elements for $+$ and $*$, respectively. For $A \subseteq
\Nn$, we use $\vec{1}_{A}$ to denote the {\bf indicator function of $A$}, i.e.
$\vec{1}_{A}(k) = 1$ if $k \in A$; and $\vec{1}_A(k) = 0$ otherwise. We
write $\vec{1}$ for $\vec{1}_{\Nn}$, $\vec{1}_p$ for $\vec{1}_{\{p^k \colon
k \ge 0\}}$ and $e_n$ for $\vec{1}_{\{n\}}$ ($n \in \Nn$). Since most of
the time we will consider the convolution product, we often simply write
$fg$ for $f*g$ and $f^k$ ($k \in \Nn$) for the $k$-th power of $f$ with
respect to the convolution product. For a nonzero arithmetic function $f$,
$f^0$ is understood to be $1$. Unless otherwise stated, by $\cA$ we mean the
$\Cc$-algebra $(\cA,+,*)$. However, we do also consider the structure $(\cA,
+, \cdot$) where $\cdot$ is the pointwise multiplication of functions. This
structure is also a $\Cc$-algebra but this time $\alpha \in \Cc$ is identified
with the constant function $n \mapsto \alpha$ $(n \ge 1)$.

For $k \in \Nn$, let $\varepsilon_k$ be the {\bf $k$-th coordinate map}, i.e.
$\varepsilon_k(f) = f(k)$ ($f \in \cA$).  Among the coordinate maps only
$\varepsilon := \varepsilon_1$ is a $\Cc$-algebra homomorphism from $\cA$
to $\Cc$. For $X \subseteq \Cc$, let
\begin{align*}
  \cA_{X} = \varepsilon^{-1}(X) = \{f \in \cA \colon f(1) \in X\}.
\end{align*}
We write $\cA_{\alpha}$ for $\cA_{\{\alpha\}}$. One sees that $\cA_0$ is the
unique maximal ideal of $\cA$ by checking that its complement is the group
of units of $\cA$.

The {\bf support} of an arithmetic function $f$, denoted by $\supp f$, is the
set of natural numbers $n$ such that $f(n) \neq 0$.  The {\bf order} of $f$,
denoted by $v(f)$, is the least element of its support if $f \neq 0$ and
is $\infty$ if $f = 0$. A {\bf prime divisor} of a set of natural numbers
$A$ is a prime that divides some member of $A$. Following the notation
in~\cite{aids}, we use $[A]$ to denote the set of prime divisors of $A$.
We say that $A$ is {\bf (multiplicatively) finitely generated} if $[A]$ is
finite. We use $\cT$ and $\cS$ to denote the subalgebras of $\cA$ consisting
of arithmetic functions with finite support and finitely generated support,
respectively. Note that $\cT$ is the $\Cc$-subalgebra of $\cS$ generated by
the $e_n$ ($n \in \Nn$).

\begin{lem}
  \label{l:prod_order} Let $f_1,\ldots, f_n \in \cA$ and $a_1,\ldots a_n$ be
  real numbers such that $0 < a_i \le v(f_i)$ for each $1 \le i \le n$. Then
  \begin{align}
    \label{eq:eval} (f_1*\cdots *f_n)\left( \prod_{i=1}^n a_i \right) =
    \prod_{i=1}^n f_i(a_i).
  \end{align}
\end{lem}
\begin{proof}
   First, if some $f_i = 0$, then both sides of~\eqref{eq:eval} are 0. So
   let us assume the order of each $f_i$ is finite. For $a \in \Rr$, we
   have
   \begin{align}
     \label{eq:sum} (f_1*\cdots * f_n)(a) = \sum_{\begin{smallmatrix}
     d_1\cdots d_n=a \\ d_i
       \in \Nn \end{smallmatrix}} f_1(d_1)\cdots f_n(d_n).
   \end{align}
   The summand $f_1(d_1)\cdots f_n(d_n)$ appears in~\eqref{eq:sum} can be
   nonzero only if $d_i \ge v(f_i) (\ge a_i)$ for each $i$. So by taking $a =
   a_1\cdots a_n$, we see that $f_1(d_1)\cdots f_n(d_n) \neq 0$ if and only
   if $d_i = v(f_i) = a_i$ for each $i$. Thus either $a_i < v(f_i)$ for some
   $i$, in that case both sides of~\eqref{eq:eval} are zero, or else $a_i =
   v(f_i)$ for each $i$, in that case both sides of~\eqref{eq:eval} equal
   $f_1(v(f_1))\cdots f_n(v(f_n))$.
\end{proof}

\begin{prop}
  \label{p:det_val} Let $f_{ij}$ be arithmetic functions $(1 \le i,j \le n)$.
  Suppose $a_i,b_i$ $(1 \le i \le n)$ are positive real numbers such that
  $a_ib_j \le v(f_{ij})$ for $1 \le i,j\le n$. Then
  \begin{align}
    \label{eq:det_val}
    \det\left(f_{ij}\right)\left( \prod_{k=1}^n a_kb_k \right) = \det\left(
    f_{ij}(a_ib_j) \right).
  \end{align}
\end{prop}
\begin{proof}
  For each permutation $\xi$ of $\{1,\dots, n\}$, by Lemma~\ref{l:prod_order}
  we have
   \begin{align*}
     \left(\sgn(\xi)\prod_{k=1}^n f_{k\xi(k)} \right)\left( \prod_{k=1}^n
     a_kb_k \right) &= \left(\sgn(\xi)\prod_{k=1}^n f_{k\xi(k)} \right)\left(
     \prod_{k=1}^n a_kb_{\xi(k)} \right)\\ &=\sgn(\xi)\prod_{k=1}^n
     f_{k\xi(k)}(a_kb_{\xi(k)}).
  \end{align*}
  Equation~\eqref{eq:det_val} now follows by summing through the permutations.
\end{proof}

Let $\norm{f}$ denote the reciprocal of $v(f)$ with the convention $1/\infty
=0$. The assignment $f \mapsto \norm{f}$ is a non-archimedean norm on $\cA$.
In particular, $\norm{f*g} = \norm{f}\norm{g}$ and consequently $\cA$ is an
integral domain. The ring operations of $\cA$ are continuous with respect to
the (ultra)-metric induced by this norm. A sequence $(f_n)$ of arithmetic
functions converges to an arithmetic function $f$, written as $f_n \to f$,
if and only if the sequence of rational numbers $(\norm{f_n -f})_n$ converges
to $0$. Note also that a map from $\cA$ to itself is continuous if and only
if it preserves convergence of sequences. Since the norm under consideration
is non-archimedean, the series $\sum_k^\infty f_k$ converges if and only if
$f_k \to 0$. In particular, for any formal power series $\sum \alpha_k X^k$
over $\Cc$ and $g \in \cA$, the series $\sum \alpha_k g^k$ converges if and
only if $\norm{g} < 1$ or equivalently $g \in \cA_0$.

The map defined by
\begin{align*}
  f \longmapsto \Exp(f) = \sum_{k =0}^\infty \frac{f^k}{k!}
\end{align*}
is a continuous isomorphism of groups from $(\cA_0, +)$ to $(\cA_1,
*)$~\cite[Theorem~2.20]{ant}. We extend it to the {\bf exponential map}
on $\cA$ by,
\begin{align*}
 f \longmapsto \exp(f(1))*\Exp(f-f(1))
\end{align*}
where $\exp$ denotes the exponential map of $\Cc$. This extension is still a
continuous group homomorphism from $(\cA,+)$ to $(\cA^{\times}, *)$ but no
longer injective since it extends the complex exponentiation. However, its
restriction to $\cA_{\Rr}$, as shown by Rearick in~\cite{rearick}, is indeed a
continuous group isomorphism from $(\cA_{\Rr},+)$ to $(\cA_+,*)$ where $\cA_+$
is the inverse image of the set of positive reals under $\varepsilon$. The
inverse of this group isomorphism, known as the {\bf Rearick logarithm},
is also continuous and we denote it by $\Log$. For convenience, we understand
$\Exp^0 = \Log^0$ as the identity map of $\cA$; and for $k \ge 1$, $\Exp^{-k} =
\Log^k$. For any $f \in \cA$, there is a largest $k \ge 0$ such that $\Log^k f$
is defined: $k=0$ if $f \notin \cA_{+}$, otherwise $k \ge 1$ is the integer
such that $\log^{k}(f(1)) \le 0$ (here $\log$ is the real logarithm). For a
nonempty $W \subseteq \cA$, let $k_W$ be the largest non-negative integer,
such that $\Log^{k_W} f$ is defined for each $f \in W$. We write $\Exp^* W$
for the set
\begin{align*}
  \{\Exp^m f \colon f \in W, m \ge -k_W\}.
\end{align*}

The ring of arithmetic functions is isomorphic, as a $\Cc$-algebra, to the ring
of formal Dirichlet series~\cite[\S~4.6]{hint} via
\begin{align}
  \label{eq:isods} f \longleftrightarrow F(s) = \sum_{n \in \Nn}
  \frac{f(n)}{n^s}.
\end{align}
Under this isomorphism $\vec{1}$ is identified with $\sum 1/n^s$ the Dirichlet
series of the Riemann zeta function $\zeta(s)$. In general, for $A \subseteq
\Nn$, $\vec{1}_A$ is identified with the Dirichlet series $\sum_{n \in \Nn}
\vec{1}_A(n)/n^s$ which converges on a proper right half-plane and extends
to a meromorphic function on $\Cc$. We call this function the {\bf zeta
function of $A$} and denote it by $\zeta_A(s)$.

The ring of arithmetic functions is also isomorphic, as a $\Cc$-algebra,
to the formal power series ring over $\Cc$ in countably many variables $t_p$
($p \in \Pp$) via
\begin{align}
  \label{eq:isops} f \longleftrightarrow F(\vec{t}) = \sum_{n \in \Nn}
  f(n)\prod_p t_p^{v_p(n)},
\end{align}
where $v_p(m)$ is the exponent of $p$ in the prime factorization of $m$. Under
this isomorphism $e_p$ is mapped to the variable $t_p$. The isomorphism
in~\eqref{eq:isops} was utilized by Cashwell and Everett in~\cite{rntf}
to show that $\cA$ is a unique factorization domain.

By a {\bf derivation} of $\cA$ we mean a $\Cc$-linear map from $\cA$ to itself
satisfying the Leibniz rule: $D(f*g) = Df*g + f*Dg$. For simplicity, we do
not distinguish by notation a derivation of $\cA$ and its unique extension to,
$\cF$, the field of fractions of $\cA$. Let $\Delta$ be a set of derivations
of $\cA$. By the {\bf kernel} of $\Delta$, written as $\ker \Delta$, we mean
the intersection of the kernels of its members. By $\ker_{\cF} \Delta$ we
mean the same but regard the members of $\Delta$ as derivations of $\cF$. In
particular, $\ker \emptyset$ and $\ker_{\cF} \emptyset$ are $\cA$ and $\cF$,
respectively. It is routine to check that $\ker_{\cF} \Delta$ is a subfield
of $\cF$ extending $\Cc$ whose intersection with $\cA$ is $\ker \Delta$.

The {\bf log-derivation} of $\cA$, denoted by $\partial_L$, is the map
sending $f \in \cA$ to the arithmetic function defined by
\begin{align*}
  (\partial_L f)(n) = \log(n)f(n).
\end{align*}
Under the isomorphism in~\eqref{eq:isods} $\partial_L$ corresponds to the
derivation $-d/ds$. For each prime $p$, the {\bf $p$-basic derivation}
of $\cA$, denoted by $\partial_p$, is the map sending $f \in \cA$ to the
arithmetic function defined by
\begin{align*}
  (\partial_p f)(n) = f(np)v_p(np).
\end{align*}
Under the isomorphism in~\eqref{eq:isops} $\partial_p$ corresponds to
$\partial/\partial t_p$ the partial derivation with respect to $t_p$. A
derivation of $\cA$ is {\bf basic} if it is $\partial_p$ for some $p$. The
kernel of $\partial_L$ is $\Cc$ and the kernel of $\partial_p$ consists of
arithmetic functions that vanish on the multiples of $p$. In other words,
\begin{align}
  \label{eq:supp_ker}
  f \in \ker \partial_p \iff p \notin [\supp f].
\end{align}
Thus the kernel of the set of basic derivations is also $\Cc$. Basic derivations
and the log-derivation are continuous. For a nice characterization of continuous
derivations of $\cA$, we refer the reader to~\cite[Section~4]{conv}. We consider
continuous derivations because the derivative of a power series with respect to
a continuous derivation can be computed term-by-term:

\begin{lem}
  \label{l:contd} Suppose $D$ is a continuous derivation of $\cA$ and $g \in
  \cA_0$. Then for any formal power series $\sum_{k=0}^{\infty} \alpha_k X^k$
  over $\Cc$,
  \begin{align*}
    D\left( \sum_{k=0}^\infty \alpha_k g^k \right) = \left(\sum_{k=1}^{\infty}
    ka_kg^{k-1}\right)* Dg.
  \end{align*}
\end{lem}
\begin{proof}
  Since $D$ is $\Cc$-linear and satisfies the Leibniz rule, for each $n
  \in \Nn$,
  \begin{align}
    \label{eq:partial_sum} D\left( \sum_{k=0}^n \alpha_k g^k \right) =
    \left(\sum_{k=1}^n k\alpha_k g^{k-1}\right)*Dg.
  \end{align}
  The left-side of~\eqref{eq:partial_sum} converges to $D\left(
  \sum_{k=0}^\infty \alpha_kg^k \right)$ by continuity of $D$. Since $g
  \in \cA_0$ and the convolution product is continuous, the right-side
  of~\eqref{eq:partial_sum} converges to $\left( \sum_{k=1}^\infty
  k\alpha_kg^{k-1}\right)*Dg$. The lemma now follows from the uniqueness
  of limits.
\end{proof}

\begin{prop}
  \label{p:Dlog} For any continuous derivation $D$ of $\cA$ and $f \in \cA$,
  $D(\Exp(f)) = \Exp(f)*Df$.
\end{prop}
\begin{proof}
  By applying Lemma~\ref{l:contd} to the series $\sum_{k=0}^{\infty} X^k/k!$
  we conclude that $D\Exp(f) = \Exp(f)*Df$ for any $f \in \cA_0$. In general,
  since $\ker D \supseteq \Cc$, it follows that for $f \in \cA$,
   \begin{align*}
     D\Exp(f) &= D(\exp(f(1))*\Exp(f-f(1)))\\ &= \exp(f(1))*D(\Exp(f-f(1)))\\
     &= \exp(f(1))*\Exp(f-f(1))*D(f-f(1))\\ &= \Exp(f)*Df.
   \end{align*}
\end{proof}

\begin{cor}
  \label{c:EL-KerD} Suppose $\Delta$ is a set of continuous derivations of
  $\cA$. Then $f \in \ker\Delta$ if and only if $\Exp(f) \in \ker\Delta$.
  Moreover, if $f \in \cA_{+}$ then $f \in \ker \Delta$ if and only if $\Log
  f \in \ker\Delta$.
\end{cor}
\begin{proof}
  By Proposition~\ref{p:Dlog}, $D(\Exp(f)) = Df * \Exp(f)$ for any $D \in
  \Delta$. Since $\Exp(f) \neq 0$, the first assertion follows. The second
  assertion follows from the first because for $f \in \cA_{+}$, $f =
  \Exp(\Log(f))$.
\end{proof}

\begin{prop}
  \label{p:der_Exp} Suppose $f_1,\ldots, f_n \in \cA$ and $D_1,\ldots, D_n$ are
  continuous derivations of $\cA$. Then for any $k \in \Zz$ such that $\Exp^k
  f_i$ is defined for all $1 \le i \le n$,
  \begin{align*}
    \det(D_j f_i) = 0 \iff \det(D_j \Exp^{k} f_i) = 0.
  \end{align*}
\end{prop}
\begin{proof}
  It suffices to show that for any $g_1,\ldots, g_n \in \cA$, $\det\left(
  D_j g_i \right) = 0$ if and only if $\det\left( D_j \Exp g_i \right)
  = 0$. But this follows immediately from Proposition~\ref{p:Dlog}, since
  \begin{align*}
    \det\left( D_j \Exp g_i \right) = \det\left( \Exp(g_i)*D_j g_i \right)
    = \det\left( D_j g_i \right)\prod_{i = 1}^n \Exp(g_i)
  \end{align*}
  and $\Exp g \neq 0$ for any $g \in \cA$.
\end{proof}

As another application of Proposition~\ref{p:Dlog}, let us compute the
function $\kappa := \Log \vec{1}$. On the one hand, $\partial_L\vec{1} =
\partial_L \kappa *\vec{1}$. On the other hand,
\begin{align*}
   \partial_L \vec{1}(n) = \log(n) = \sum_{p|n} v_p(n)\log p =  \sum_{p^j |
   n} \log p = (\Lambda * \vec{1})(n).
\end{align*}
So $\partial_L \kappa = \Lambda$ is the von Mangoldt's function. Thus,
\begin{align*}
  \kappa(n) =
  \begin{cases}
    \dfrac{\Lambda(p^j)}{\log(p^j)} = \dfrac{1}{j} &\text{if $n = p^j$ for some
    prime $p$ and $j \ge 1$;} \\ 0 &\text{otherwise.}
  \end{cases}
\end{align*}

For $g \in \cA$, let $\mf_g$ denote the $\Cc$-linear map from $\cA$
to itself defined by $\mf_g(f) = g \cdot f$ (pointwise product). It
is clear that $\norm{\mf_g(f)} \le \norm{f}$. Thus, $\mf_g$ preserves
null sequences and hence is continuous by linearity. It is also clear
that $\mf_h$ is the compositional inverse of $\mf_g$ if and only if $h$
is the pointwise multiplicative inverse of $g$. If $g$ is {\bf completely
additive}, i.e. $g(nm) = g(n) + g(m)$ for all $n,m \in \Nn$, one checks that
$\mf_g$ is a (continuous) derivation of $\cA$ and vice versa. For example,
$\mf_{\log}$ is simply the log-derivation $\partial_L$. We will use the more
suggestive notation $\partial_g$ for $\mf_g$ in case it is a derivation. A
completely additive function is determined by its action on the primes
and its value at 1 must be $0$. Besides the real logarithm, the $p$-adic
valuation $v_p$, and the function $\Omega$, which counts (with multiplicity)
the total number of prime factors of its argument, are some examples of
completely additive function. If $g$ is {\bf completely multiplicative},
i.e. $g \neq 0$ and $g(nm) = g(n)g(m)$ for all $n,m \in \Nn$, one checks
that $\mf_g$ is a nonzero (continuous) $\Cc$-algebra endomorphism of $\cA$
and vice versa. If, in addition, $g$ vanishes nowhere then its pointwise
multiplicative inverse is also completely multiplicative. Thus $\mf_g$ is
a continuous automorphism of $\cA$. For example, $\mf_{\mathbf{I}}$, where
$\mathbf{I}$ is the identity map of $\Nn$, is a continuous automorphism of
$\cA$. A completely multiplicative function is determined by its action on
the primes and its value at $1$ must be $1$. Besides the identity function,
the map $n \mapsto n^\alpha$ ($\alpha \in \Cc$) and $\vec{1}_p$ are some
examples of completely multiplicative functions.

We conclude this section by an observation that will be used a number of
times in Section~\ref{sec:mg-ind}.
\begin{lem}
  \label{l:ord_inv} For any $f,g \in \cA$, $p \in \Pp$ and $i \in \Zz$ such
  that $\mf_g^i$ is defined, $v(\partial_p f) \le v(\partial_p \mf_g^i(f))$.
  Moreover, the equality holds if $g(m) \neq 0$ for all $m >1$.
\end{lem}
\begin{proof}
For any $n \ge 1$,
 \begin{align}
   \label{eq:ord_pre} \partial_p \mf^i_g(f) (n) &= v_p(np)(g(np))^if(np) =
  (g(np))^i\partial_pf(n).
 \end{align}
Thus $\partial_p f(n) =0$ implies $\partial_p \mf^i_g(f)(n) = 0$ and so the
inequality in the lemma holds. Furthermore, if $g(np) \neq 0$ for all $n$,
the reverse implication is also true. Thus, $\partial_p f$ and $\partial_p
\mf_g^i f$ must have the same order.
\end{proof}

\section{Ax's Theorem for $\cA$} 
\label{sec:Ax4A} 

Our main observation is simple: Ax's Theorem holds for $(\cA, + , *)$.
\begin{thm}
  \label{th:sa} Suppose $\cC = \ker_{\cF} \Delta$ for some set $\Delta$
  of continuous derivations of $\cA$. Let $f_1,\dots, f_n$ be arithmetic
  functions such that either
  \begin{enumerate}[{\upshape(}$1${\upshape)}]
    \item \label{i:power} $\Exp(f_1),\dots, \Exp(f_n)$ are multiplicatively
      independent modulo $\cC^{\times}$; or
  \item \label{i:lin-ind} $f_1,\ldots, f_n$ are $\Qq$-linearly independent
  modulo $\cC$.
  \end{enumerate}
  Then
  \begin{align*}
    \td_{\cC}\cC(f_1,\dots,f_n, \Exp(f_1),\dots, \Exp(f_n)) \ge n +
    \rk(Df_i)_{\begin{smallmatrix} D \in \Delta \\ 1 \le i \le n
    \end{smallmatrix}}
  \end{align*}
\end{thm}
\begin{proof}
  Take the field $F$ in Theorem~\ref{th:genax} to be $\cF$ and $C =
  \cC = \ker_{\cF}\Delta$. Let $y_i = f_i$ and $z_i = \Exp f_i$ ($i
  =1,\dots,n$). Then by Proposition~\ref{p:Dlog}, $Dy_i = Dz_i/z_i$
  for all $D \in \Delta$ and $1 \le i \le n$. Therefore, Condition~(a) in
  Theorem~\ref{th:genax} holds. Condition~\eqref{i:power} and~\eqref{i:lin-ind}
  now translate into Condition~(b) and~(b$^\prime$) in Theorem~\ref{th:genax},
  respectively and so the inequality about the transcendence degree follows.
\end{proof}

As our first illustration of the power of Ax's theorem, we use it to deduce the
following generalization of Theorem~5.3 of~\cite{aids}. For $f \in \cA_+$ and
$g \in \cA$, we write $f^g$ as a shorthand for the function $\Exp(g*\Log f)$.
\begin{thm}
  \label{th:Log_power} Let $\Delta$ be a set of continuous derivations of $\cA$
  and $\cC = \ker_{\cF} \Delta$. Suppose $f \in \cA_+ \setminus \ker \Delta$
  and $1=c_0, c_1, \ldots, c_n \in \ker\Delta$ are linearly independent over
  $\Qq$, then $\Log f, f=f^{c_0}, f^{c_1},\ldots, f^{c_n}$ are algebraically
  independent over $\cC$.
\end{thm}
\begin{proof}
  By Corollary~\ref{c:EL-KerD}, $\Log f \notin \ker \Delta$. Thus $D_0
  \Log f \neq 0$ for some $D_0 \in \Delta$. We claim that $f = f^{c_0}$,
  $f^{c_1},\ldots, f^{c_n}$ are multiplicatively independent modulo
  $\cC^{\times}$. Suppose not, then there exist integers $k_0,\ldots, k_n$
  not all zeros such that
  \begin{align*}
    f^{k_0}f^{k_1c_1}\cdots f^{k_nc_{n}} =\Exp\left( (k_0 + k_1c_1+ \dots +
    k_nc_n)\Log f \right)
  \end{align*}
  belongs to $\cC \cap \cA = \ker \Delta$. An application of
  Corollary~\ref{c:EL-KerD} yields $(k_0 + k_1c_1 + \dots +k_nc_n) \Log f
  \in \ker \Delta$. In particular,
  \begin{align*}
    0 &= D_0( (k_0 + k_1c_1 + \cdots + k_nc_n)\Log(f)) \\ &= (k_0 + k_1c_1 +
    \cdots + k_{n}c_{n})D_0(\Log f).
  \end{align*}
  Since $D_0(\Log f) \neq 0$, that means $k_0+k_1c_1 + \dots + k_nc_n$
  must be zero, contradicting the assumption that $1,c_1,\dots, c_n$ are
  $\Qq$-linearly independent. This establishes the claim. Now by applying
  Theorem~\ref{th:sa} to the $n+1$ functions $c_i\Log f$ ($0 \le i \le n$),
  we conclude that the transcendence degree of the field
  \begin{align*}
    \cC(c_i\Log f,f^{c_i})_{0 \le i \le n} = \cC(\Log
      f,f,f^{c_1},\ldots, f^{c_n})
  \end{align*}
  over $\cC$ is at least 
  \begin{align*}
    (n+1) + \rk_{\cF}(D\Log f, c_iD\Log f)_{\begin{smallmatrix} D \in \Delta \\
      1 \le i \le n \end{smallmatrix}}.
  \end{align*}
  Since $D_0\Log f \neq 0$, the rank appeared above is $1$. This establishes
  the algebraic independence of $\Log f,f, f^{c_i}$ ($1 \le i \le n$)
  over $\cC$.
\end{proof}

\begin{cor}
  With the same notation as in Theorem~\ref{th:Log_power}, $\Log f$ is
  transcendental over $\cC(f,f^{c_1},\ldots, f^{c_n})$ for any $c_1,\ldots,
  c_n \in \ker \Delta$.  \label{c:trans}
\end{cor}
\begin{proof}
  By re-indexing, if necessary, $1=c_0, c_1,\ldots, c_m$ (for some $0
  \le m \le n)$ form a basis of the $\Qq$-span of $1,c_1,\ldots, c_n$. By
  Theorem~\ref{th:Log_power}, $\Log f$ is transcendental over $\cC(f,f^{c_1},
  \ldots, f^{c_m})$. Since each $c_i$ ($0 \le i \le n$) is a $\Qq$-linear
  combination of $1,c_1,\ldots, c_m$, each $f^{c_i}$ is algebraic over
  $\cC(f,f^{c_1},\ldots, f^{c_m})$ and so the corollary follows.
\end{proof}

The following corollary is a very special case of Corollary~\ref{c:trans}. We
refer the reader to~\cite[Section~5]{aids} for its consequences.
\begin{cor}
  \label{c:log_zeta_power} For any complex numbers $c_1,\dots,c_n$, $\log
  \zeta$ is transcendental over $\Cc(\zeta^{c_1},\dots,\zeta^{c_n})$. In
  particular, $\log \zeta$ is transcendental over $\Cc(\zeta)$.
\end{cor}
\begin{proof}
  By invoking the isomorphism in~\eqref{eq:isods}, it suffices to show that
  the function $\kappa = \Log \vec{1}$ is transcendental over $\Cc(\vec{1},
  \vec{1}^{c_1},\ldots, \vec{1}^{c_n})$ but that follows immediately from
  Corollary~\ref{c:trans} by taking $\Delta = \{\partial_L\}$ and $f =
  \vec{1}$.
\end{proof}

The central result about algebraic independence of arithmetic functions is the
following criterion of Shapiro and Sparer~\cite[Theorem~3.1]{aids}. We refer the
reader to~\cite{aids,imaf} and~\cite{obs} for its numerous applications.
\begin{thm*}
  \label{th:Jac} Let $f_1,\ldots, f_n$ be arithmetic functions. Suppose $D_1,
  \ldots, D_n$ are derivations of $\cA$ such that $\det(D_jf_i) \neq 0$ then
  $f_1, \ldots, f_n$ are algebraically independent over $\ker\{D_1,\ldots,
  D_n\}$.
\end{thm*}
As our second illustration of the power of Ax's Theorem, we use it to
strengthen the Jacobian criterion when the derivations involved are continuous.
\begin{thm}
  \label{th:gJac} Let $f_1, \ldots, f_n \in \cA$. Suppose $D_1,\ldots, D_n$
  are continuous derivations of $\cA$ such that $\det\left( D_j f_i \right)
  \neq 0$ then the set of arithmetic functions
  \begin{align*}
    \Exp^*\{f_i \colon 1 \le i \le n\}
  \end{align*}
  is algebraically independent over $\ker\{D_1,\dots, D_n\}$.
\end{thm}
\begin{proof}
  Let $\cC=\ker_{\cF}\{D_1,\ldots, D_n\}$ and $k_0 \ge 0$ be the largest
  integer such that for each $1 \le i \le n$, $g_i := \Log^{k_0} f_i$
  is defined. It then suffices to show that for any $m \ge 1$, the set of
  arithmetic functions
  \begin{align*}
    \{\Exp^k g_i \colon 0 \le k \le m, 1 \le i \le n\}
  \end{align*}
  is algebraically independent over $\cC \supseteq \ker\{D_1,\ldots, D_n\}$. We
  will prove this by induction on $m$. First, let us argue that $g_1,\ldots,
  g_n$ are $\Qq$-linearly independent modulo $\cC$. Suppose some $\Qq$-linear
  combination $\sum r_ig_i$ of the $g_i$'s belongs to $\cC$ then by applying
  $D_j$ ($1 \le j \le n$) to the linear combination we obtain a system of $n$
  linear equations:
  \begin{align*}
    \sum_{i=1}^n r_i D_j g_i = 0 \qquad (1 \le j \le n).
  \end{align*}
  Since $\det\left( D_j f_i \right) \neq 0$, by Proposition~\ref{p:der_Exp}
  $\det\left( D_j g_i \right) \neq 0$ as well and so the $r_i$ ($1 \le i
  \le n$) must be all zero.  This establishes the claim. Now we can apply
  Theorem~\ref{th:sa} to $g_1,\ldots, g_n$ and conclude that \begin{align*}
    \td_{\cC}\cC(g_1,\dots, g_n, \Exp(g_1),\dots, \Exp(g_n)) \ge n +
    \rk(D_j g_i).
  \end{align*}
  Again since $\det(D_jg_i) \neq 0$, the $\cF$-rank of $(D_j g_i)$ is
  $n$. This establishes the algebraic independence of $g_i, \Exp(g_i)$
  ($1 \le i \le n$) over $\cC$, i.e. the case $m=1$ of the theorem.

  For the induction step, suppose the functions $\Exp^k(g_i)$ $(0 \le k \le m,
  1 \le i \le n)$ are algebraically independent over $\cC$ for some $m \ge
  1$. In particular, these functions are $\Qq$-linearly independent modulo
  $\cC$ and we conclude from Theorem~\ref{th:sa} that the transcendence
  degree of the field
  \begin{align*}
    \cE: = \cC&(\Exp^k(g_i) \colon 0 \le k \le m+1, 1 \le i \le n)
  \end{align*}
  over $\cC$ is at least $n(m+1) + \rk V$ where $V$ is the set of vectors
  \begin{align*}
    \{ (D_j \Exp^k(g_i))_{1 \le j \le n} \colon 0 \le k \le m, 1 \le i \le n\}.
  \end{align*}
  Again because $\det\left( D_j g_i \right) \neq 0$, the $\cF$-rank of $V$
  is at least (in fact exactly) $n$. Consequently, the transcendence degree
  of $\cE$ over $\cC$ is $(m+2)n$. This establishes the induction step and
  hence the theorem.
\end{proof}

Theorem~\ref{th:gJac}, strictly speaking, is not a generalization of
the Jacobian criterion because it requires the derivations involved to be
continuous. However, to the best of our knowledge, all existing applications
of this criterion involve only the log-derivation and the basic derivations so
to all of them Theorem~\ref{th:gJac} is applicable. In the next two sections,
we will generalize a number of results in~\cite{aids, imaf} and~\cite{obs}
in various directions.

\section{Algebraic Independence} \label{sec:algdep} 

We begin this section with a very special case of Theorem~\ref{th:gJac}
when only a single derivation is involved.
\begin{prop}
  \label{p:keracl} Let $D$ be a continuous derivation of $\cA$ and $f \notin
  \ker D$. Then $\Exp^*\{f\}$ is algebraically independent over $\ker D$.
  In particular, $\ker D$ is algebraically closed in $\cA$.
\end{prop}
Proposition~\ref{p:keracl} generalizes Proposition~2.1 of~\cite{imaf}. For
example, by taking $D = \partial_L$, one sees that $\Cc$ is algebraically
closed in $\cA$ and that $\Log(f), f, \Exp(f)$ are algebraically
independent over $\Cc$ for $f \in \cA_+\setminus \Cc$. We should point
out that the kernel of a derivation of $\cA$, whether continuous or not,
is always algebraically closed in $\cA$. As a matter of fact, the argument
given for that in~\cite{aids} (see Lemma~2.1 of~\cite{aids}) works for
any characteristic zero integral domain. From Proposition~\ref{p:keracl},
we can also deduce the following generalization of Theorem~2.1 of~\cite{aids}.
\begin{thm}
  \label{th:fgsa} Let $f \in \cA$ and $(g_i)_{i\in I}$ be a family of
  arithmetic functions. Suppose
  \begin{align*}
    [\supp f] \not\subseteq \bigcup_{i \in I} [\supp g_i]
  \end{align*}
  then $\Exp^*\{f\}$ is algebraically independent over the subalgebra of $\cA$
  generated by the $g_i$ $(i \in I)$.
\end{thm}
\begin{proof}
  By the assumption there is a prime $p \in [\supp f]$ that is not in the
  union of the $[\supp g_i]$ ($i \in I$). So by Proposition~\ref{p:keracl},
  $\Exp^*\{f\}$ is algebraically independent over $\ker \partial_p$ which
  contains the subalgebra of $\cA$ generated by the $g_i$ ($i \in I$).
\end{proof}
We provide a proof of one of the many consequences of~\cite[Theorem~2.1]{aids}.
The reader can consult~\cite[p.697-699]{aids} for the others.
\begin{cor}
  \label{c:S_is_acl} $\cS$ is algebraically closed in $\cA$.
\end{cor}
\begin{proof}
  Suppose $g_1, \ldots, g_n \in \cS$ and $f \in \cA \setminus \cS$. Then
  $[\supp f]$ is infinite while the union of $[\supp g_i]$ ($1 \le i \le n$)
  is finite.  So it follows from Theorem~\ref{th:fgsa} that $\Exp^*\{f\}$, in
  particular $f$ itself, is algebraically independent over $\Cc[g_1,\ldots,
  g_n]$. Since $g_i \in \cS$ ($1 \le i \le n$) are taken arbitrarily, we
  conclude that $f$ is algebraically independent over any finitely generated
  subalgebra of $\cS$ and hence over $\cS$ itself.
\end{proof}
\begin{exa}
  \label{ex:1_tran_T} The function $\vec{1}$ is not a member of $\cS$ so by
  Corollary~\ref{c:S_is_acl} it is transcendental over $\cS$ and hence over
  $\cT$. In terms of Dirichlet series, that means the Riemann zeta function
  is transcendental over the subalgebra of {\bf Dirichlet polynomials}
  (Dirichlet series with only finitely many nonzero terms).
\end{exa}
In contrast, $\cT$ is not algebraically closed in $\cA$ (in fact, not
even in $\cS$).  For instance, $\vec{1}_2 = \sum_{k=0}^\infty e_2^k$ is in
$\cS\setminus \cT$ but it is algebraic over $\cT$ since its inverse $1-e_2$ is
in $\cT$. This shows, in particular, that the algebra of Dirichlet polynomials
is not algebraically closed in the algebra of convergent Dirichlet series.

\begin{thm}
  \label{th:tri} Let $f_1,\dots, f_n$ be arithmetic functions. Suppose there
  exist $D_1,\dots, D_n$ continuous derivations of $\cA$ such that for each
  $1 \le i < j \le n$,
  \begin{align*}
    f_i \in \ker D_j \setminus \ker D_i.
  \end{align*}
  Then the set of arithmetic functions $\Exp^*\{f_1,\dots, f_n\}$ is
  algebraically independent over $\ker\{D_1,\dots, D_n\}$.
\end{thm}
\begin{proof}
  It is an immediate consequence of Theorem~\ref{th:gJac}. This is because the
  assumption implies $\left( D_j f_i \right)$ is a lower triangular matrix with
  non-zero entries on its diagonal hence $\det \left( D_j f_i \right) \neq 0$.
\end{proof}

\begin{cor}
  \label{c:tri_supp} Let $f_1,\dots, f_n \in \cA$. Suppose there exist
  $p_1,\ldots, p_n$ such that for each $1 \le i < j \le n$
  \begin{align*}
    p_j \in [\supp f_j] \setminus [\supp f_i],
  \end{align*}
  then the set of arithmetic functions $\Exp^*\{f_1,\dots, f_n\}$ is
  algebraically independent over the kernel of $\{\partial_{p_1},\dots,
  \partial_{p_n} \}$.
\end{cor}
\begin{proof}
  Take $D_j$ in Theorem~\ref{th:tri} to be $\partial_{p_j}$ ($1 \le i \le n$).
\end{proof}

\begin{exa}
  \label{ex:e_p} Let $p_1,\dots, p_n$ be distinct primes. By taking
  $f_i = e_{p_i}$ ($1 \le i \le n-1$) and $f_n = \vec{1}_\Pp$ in
  Corollary~\ref{c:tri_supp}, the $\Cc$-algebraic independence of
  $\Exp^*\{e_{p_1},\ldots, e_{p_{n-1}}, \vec{1}_\Pp\}$ follows. Moreover,
  since $n$ is arbitrary, that means the set of arithmetic functions
  \begin{align*}
    \Exp^*(\{e_p \colon p \in \Pp\} \cup \{\vec{1}_\Pp\}),
  \end{align*}  
  is algebraically independent over $\Cc$.
\end{exa}

\begin{exa}
  \label{ex:two} Corollary~\ref{c:tri_supp} generalizes Lemma~3 of~\cite{obs}:
  Suppose $f_1,f_2 \in \cA \setminus \Cc$ with $[\supp f_1] \neq [\supp f_2]$.
  Without loss of generality, there is a prime $p_2 \in [\supp f_2]$ but not in
  $[\supp f_1]$. Since $f_1$ is not in $\Cc$, there exists a prime $p_1 \in
  [\supp f_1]$. Thus Corollary~\ref{c:tri_supp} implies $\Exp^*\{f_1,f_2\}$ is
  algebraically independent over $\Cc$. In particular, if $F(s)= \sum
  \alpha_n/n^s$ is a non-constant formal Dirichlet series such that $\alpha_n=0$
  whenever $n$ is a multiple of some fix prime $p$, then $F(s)$ and $\zeta(s)$
  are algebraically independent over $\Cc$.
\end{exa}

Knowing that a function is non-vanishing at a particular point certainly implies
that it is nonzero. The following proposition is hence a corollary of
Theorem~\ref{th:gJac}. We invite the reader to prove it (or
see~\cite[Corollary~2.3]{imaf} for a proof) by checking the left-side of
Equation~\eqref{eq:gvj} expresses the value of $\det\left( \partial_{p_j} f_i
\right)$ at $m$.
\begin{prop}
  \label{p:gvj} For any $f_1,\ldots, f_n \in \cA$, if there exist distinct
  primes $p_1,\ldots, p_n$ such that for some $m \in \Nn$,
 \begin{align}
  \label{eq:gvj} \sum_{k_1\dotsm k_n =m} \left(\prod_{j=1}^n
  v_{p_j}(k_jp_j)\right) \det(f_i(k_jp_j)) \neq 0,
\end{align}
then the set of arithmetic functions $\Exp^*\{f_1,\ldots, f_n\}$
is algebraically independent over $\ker\{\partial_{p_1}, \ldots,
\partial_{p_n}\}$.
\end{prop}
By setting the $m$ in Proposition~\ref{p:gvj} to various values, one obtains
strengthened versions of Test I--IV in~\cite{imaf}. These tests were used
to establish algebraic independence of various Fibonacci and Lucas zeta
functions~\cite[Proposition~2.5, 2.6]{imaf}. We state here only the simplest
case, i.e. $m=1$.
\begin{cor}
  \label{c:ind_at_primes} Suppose $f_1,\dots, f_n$ are arithmetic functions
  such that $\det(f_i(p_j)) \neq 0$ for some primes $p_1,\ldots, p_n$. Then
  the set of functions 
  \begin{align*}
    \Exp^*\{f_i \colon 1 \le i \le n\}
  \end{align*}
  is algebraically independent over $\ker\{\partial_{p_j} \colon 1 \le j \le
  n\}$.
\end{cor}

\begin{exa}
  \label{ex:1_and_1p} For any distinct primes $p_1,\dots, p_{n}$, take $f_i =
  \vec{1}_{p_i}$ ($1 \le i \le n-1$) and $f_{n} = \vec{1}$, then
  \begin{align*}
    \det\left( f_i(p_j) \right) =
    \begin{pmatrix}
      1 & 0 &\dots & 0 \\ 0 & 1 &\dots & 0 \\
    \vdots & \vdots &\ddots & \vdots \\
      1 & 1 & \dots & 1
    \end{pmatrix}
    = 1.
  \end{align*}
  Thus by Corollary~\ref{c:ind_at_primes}, $\Exp^*\{\vec{1}_p, \vec{1}
  \colon p \in \Pp \}$ is algebraically independent over $\Cc$.
\end{exa}

\begin{exa}
  \label{ex:tau*} The function $\tau_*:=(\vec{1}-1)^2$ which counts the number
  of proper factors of its argument and $\vec{1}_\Pp$ are algebraically
  independent over $\Cc$. This is because $\partial_p \vec{1}_\Pp = 1$
  for every prime $p$, so
  \begin{align*}
    \det
    \begin{pmatrix}
      \partial_2 \tau_* & \partial_3 \tau_* \\
      \partial_2 \vec{1}_\Pp & \partial_3 \vec{1}_\Pp
    \end{pmatrix}
    & = \partial_2 \tau_* - \partial_3 \tau_*
  \end{align*}
  and its value at $4$ is $v_2(8)\tau_*(8) - v_3(12)\tau_*(12) = 2 \neq 0$.
  Note that Corollary~\ref{c:ind_at_primes} cannot be used to establish
  this fact since $\tau_*$, or more generally any member of the square of
  the maximal ideal of $\cA$, vanishes at every prime.
  \end{exa}
 
  For $f_1,\dots, f_n \in \cA$, let $\mu_{d}(\vec{f})$ be the minimum of
  $\norm{P(f_1,\dots, f_n)}$ taken over all complex polynomials $P$ of total
  degree $d$. The function $d \mapsto \mu_d(\vec{f})$ can be viewed as a
  quantitative measure of algebraic independence of $f_1,\dots, f_n$ over
  $\Cc$. Several results about this measure were proved in~\cite{imaf}. Our
  method, due to its non-constructive nature, cannot produce those results.
  However, the non-quantitative part of both Theorem~3.2 and Theorem~3.4
  of~\cite{imaf} can be generalized as follows.
\begin{thm}
  \label{th:m_j<=ord} Let $f_1,\ldots, f_n \in \cA$ and $D_1,\ldots, D_n$
  be continuous derivations of $\cA$. Suppose $m_1,\ldots, m_n \in \Nn$ such
  that $m_j \le v(D_jf_i)$ for all $1 \le i, j \le n$ and that $\det\left(
  D_jf_i(m_j) \right) \neq 0$ then the set of functions
  \begin{align*}
    \Exp^*\{f_i \colon 1 \le i \le n\}
  \end{align*}
  is algebraically independent over $\ker\{D_1,\ldots, D_n\}$.
\end{thm}
\begin{proof}
  By taking $a_i = 1$ and $b_j = m_j$ $(1 \le i,j \le n)$ in
  Proposition~\ref{p:det_val}, we conclude that the value of $\det\left(
  D_j f_i \right)$ at $m_1\cdots m_n$ is $\det\left( D_j f_i(m_j) \right)$
  which is assumed to be nonzero. The algebraic independence statement now
  follows form Theorem~\ref{th:gJac}.
\end{proof}
We can arrive to the same conclusion of Theorem~\ref{th:m_j<=ord} if $m_i \le
v(D_jf_i)$ for all $1 \le i, j \le n$: the same proof goes through by taking
$a_i = m_i$ and $b_j =1$ ($1 \le i,j \le n$). The next lemma is another easy
consequence of Proposition~\ref{p:det_val}.  The same is true, more generally,
for generalized Dirichlet series (see~\cite[Lemma~8.8]{aids}).
\begin{lem}
  \label{l:dep_at_primes} Suppose $f_1,\ldots, f_n$ are non-zero arithmetic
  functions and $p_1,\ldots, p_n$ are $n$ distinct primes such that the
  Jacobian $\det(\partial_{p_j}f_i)$ is zero then $\det(v_{p_j}(v f_i)) = 0$.
\end{lem}
\begin{proof}
  Let $m_i$ be the order of $f_i$ $(1 \le i \le n)$. Note that for $1 \le i,
  j \le n$, $0 < m_i/p_j \le v(\partial_{p_j}f_i)$. So by taking $a_i = m_i$
  and $b_i = 1/p_i$ ($1 \le i \le n$) in Proposition~\ref{p:det_val}, we have
  \begin{align*}
    \det\left( \partial_{p_j} f_i \right)\left( \prod_{k=1}^n \frac{m_k}{p_k}
    \right) &= \det\left( \partial_{p_j}f_i\left( \frac{m_i}{p_j}\right)
    \right)\\ = \det\left( v_{p_j}(m_i)f_i(m_i) \right) &= \left(\prod_{i=1}^n
    f_i(m_i) \right)\det\left( v_{p_j}(m_i) \right).
  \end{align*}
  The lemma follows since $f_i(m_i)$ is non-zero for each $i$.
\end{proof}

Lemma~\ref{l:dep_at_primes} was used to prove Theorem~7 in~\cite{obs}. It
states that a set of nonzero non-invertible arithmetic functions is
algebraically independent over $\Cc$ if the norms of its members are pairwise
relatively prime. Essentially the same proof yields a more general result:
\begin{thm}
  \label{th:no-trivial-prod} Suppose $W$ is a set of non-zero arithmetic
  functions whose orders are multiplicatively independent then $\Exp^*W$
  is algebraically independent over $\Cc$.
\end{thm}
\begin{proof}
  Suppose on the contrary that $\Exp^*W$ is algebraically dependent over $\Cc$,
  then there exist $f_1,\ldots, f_n \in W$ such that
  \begin{align*}
    \Exp^*\{f_1,\ldots, f_n\}
  \end{align*}
  is algebraically dependent over $\Cc$. So for any choice of distinct
  primes $p_1,\ldots, p_n$, $\det\left( \partial_{p_j}f_i \right) = 0$
  by Theorem~\ref{th:gJac}, as a result $\det(v_{p_j}(vf_i)) = 0$
  by Lemma~\ref{l:dep_at_primes}. That means the set of vectors
  \begin{align*}
    \left\{
      \begin{pmatrix}
      v_p(v f_1) \\ \vdots \\ v_p( vf_n)
      \end{pmatrix}
      \colon p \in \Pp
    \right\}
  \end{align*}
  has $\Qq$-rank strictly less than $n$. Since it has the same $\Qq$-rank
  as the set
  \begin{align*}
    \{(v_p(v f_i))_{p \in \Pp} \colon 1 \le i \le n\},
  \end{align*}
  there exist $k_1, \ldots, k_n \in \Zz$ not all zero such that for each
  prime $p$,
  \begin{align*}
    0 = \sum_{i=1}^n k_iv_p(v f_i) = v_p\left( \prod_{i=1}^n (v f_i)^{k_i}
    \right).
  \end{align*}
  That means $\prod_{i=1}^n (v f_i)^{k_i} = 1$ contradicting the assumption
  that the orders $v(f_i)$ ($1 \le i \le n$) are multiplicatively independent.
\end{proof}
\begin{exa}
  \label{ex:e_n} By Theorem~\ref{th:no-trivial-prod} the set
  $\Exp^*\{e_{n_1},\dots, e_{n_k}\}$ is algebraically independent over $\Cc$
  if $v(e_{n_i}) = n_i$ ($1 \le i \le n$) are multiplicatively independent. The
  converse is also true and it follows easily from the fact that $e_m*e_n
  = e_{mn}$ for any $n,m \in \Nn$. Thus for a set of natural numbers $N$,
  the necessary and sufficient condition for
  \begin{align*}
    \Exp^*\{e_n \colon n \in N\}
  \end{align*}
  to be algebraically independent over $\Cc$ is that the elements of $N$ are
  multiplicatively independent. Note that Theorem~7 in~\cite{obs} alone does
  not imply this fact since there are multiplicatively independent numbers
  such as 2 and 6 that are not relatively prime.
\end{exa}

\section{$\mf_g$-Transcendence} \label{sec:mg-ind} 

In this section we will establish some criteria for algebraic independence of
images of a single arithmetic function under operators of the form $\mf_g$.
Let $\cB$ be a subalgebra of $\cA$, we say that an arithmetic function $f$
is {\bf $\mf_g$-transcendental} over $\cB$ if $\{\mf_g^j f \colon j \in J\}$
algebraically independent over $\cB$ where $J = \Nn \cup \{0\}$ if $\mf_g$
is not invertible, otherwise $J = \Zz$.

\begin{thm}
  \label{th:[supp f]>n} Let $f, g$ be arithmetic functions. Suppose $p_1,\dots,
  p_n \in [\supp f]$ such that $g(v(\partial_{p_j} f)p_j)$ $(1 \le j \le n)$
  are distinct and nonzero. Then for any $k \ge 0$, the set of functions
  \begin{align*}
    \Exp^*\{\mf_g^i f \colon k \le i \le k+n-1\}
  \end{align*}
  is algebraically independent over $\ker\{\partial_{p_1},\dots,
  \partial_{p_n}\}$. Moreover, if $g$ is nowhere vanishing then the same is
  true for any integer $k$.
\end{thm}
\begin{proof}
  Let $f_i = \mf_g^i f$ ($k \le i \le k+n-1$) and $m_j = v(\partial_{p_j}f)$
  $(1 \le j \le n)$. By Lemma~\ref{l:ord_inv}, $m_j \le v(\partial_{p_j}
  f_i)$ for all $k \le i \le k+n-1$ and $1 \le j \le n$. So by
  Theorem~\ref{th:m_j<=ord} it suffices to show that
  \begin{align*}
    \det\left(\partial_{p_j}f_i(m_j)\right) & = \det\left(
    v_{p_j}(m_jp_j)(g(m_jp_j))^if(m_jp_j) \right)\\ &= \det\left(
    (g(m_jp_j))^{i-k} \right)\prod_j \partial_{p_j}f(m_j)(g(m_jp_j))^k
  \end{align*} 
  does not vanish. This is indeed the case because for each $j$, $m_j$ is the
  order of $\partial_{p_j}f$ hence $\partial_{p_j}f(m_j) \neq 0$ and $g(m_jp_j)
  \neq 0$ by our assumption on $g$; moreover $g(m_jp_j)$ $(1 \le j \le n)$
  are assumed to be distinct, so the last determinant is Vandermonde. Finally,
  nothing in the argument above prevents $k$ from being negative so long as
  $\mf_g^k$ is defined but that precisely requires $g$ to be nowhere vanishing.
\end{proof}

\begin{exa}
  \label{ex:1Q} Let $Q$ be a nonempty finite set of primes. Since for $q
  \in Q = [\supp \vec{1}_Q]$,
  \begin{align*}
    \log(v(\partial_q \vec{1}_Q)q) = \log(q)
  \end{align*}
  are all distinct and nonzero, it follows from Theorem~\ref{th:[supp f]>n}
  (by taking $g$ to be the real logarithm) that $\vec{1}_Q$ does not satisfy
  any differential algebraic equation with respect to $\partial_L$ of order
  less than $|Q|$ over the kernel of $\{\partial_q \colon q \in Q\}$.
\end{exa}

The assumption ``$g(v(\partial_{p_j} f)p_j)$ $(1\le j \le n)$ are distinct''
in Theorem~\ref{th:[supp f]>n} is necessary. Consider, for example, the
function $e_n$. It satisfies the following linear differential equation:
\begin{align*}
  \partial_L X - \log(n)X = 0
\end{align*}
and its support is generated by the prime divisors of $n$. Therefore,
the conclusion of Theorem~\ref{th:[supp f]>n} is false when $n$ has more
than one prime factor. Note also that for $f= e_n$ the assumption on $g$
in Theorem~\ref{th:[supp f]>n} cannot be met by any arithmetic function
since $v(\partial_p e_n)p = (n/p)p = n$ for all $p \in [\supp e_n]$. This
example also shows that the assumption ``$n_ip_i$ are distinct'' is needed
for Corollary~3.5 of~\cite{imaf}.

The following lemma is a rather simple observation about algebraic independence
of arithmetic functions over $\cS$. Since it will be called upon several times,
we include it here for the record.  For a set of primes $I$, let $\Delta_I$
be the set of basic derivations indexed by $I$, i.e. $\{\partial_p \colon
p \in I\}$. We write $\Delta_f$ for $\Delta_{[\supp f]}$.
\begin{lem}
  \label{l:cof2S} Let $I$ be a set of primes. If $E$ is a set of arithmetic
  functions that is algebraically independent over $\ker \Delta_J$ for
  any co-finite subset $J$ of $I$, then $E$ is algebraically independent
  over $\cS$.
\end{lem}
\begin{proof}
It suffices to show that $E$ is algebraically independent over every finitely
generated subalgebra of $\cS$. Suppose $\cH$ is a subalgebra of $\cS$ generated
by some $h_0,\ldots, h_d \in \cS$. Since the sets $[\supp h_i]$ ($0 \le i
\le d$) are finite so is their union $H$. Therefore, $E$, by assumption,
is algebraically independent over the kernel of $\Delta_{I\setminus H}$. We
can conclude that $E$ is algebraically independent over $\cH$ since each
derivation in $\Delta_{I \setminus H}$ kills every $h_i$ ($0 \le i \le d$).
\end{proof}

\begin{thm} \label{th:m_g-trans}
  Let $g \in \cA$ be eventually injective and $f \in \cA\setminus \cS$.
  The set of functions
  \begin{align*}
    \Exp^*\{\mf_g^i f \colon i \ge 0\}
  \end{align*}
  is algebraically independent over the kernel of any infinite subset of
  $\Delta_f$, and hence over $\cS$. In addition, if $g$ is nowhere vanishing,
  then $i$ can range through the integers.
\end{thm}
\begin{proof}
  Since $f \notin \cS$, $\Delta_f$ is infinite and so are its co-finite
  subsets. Let $J$ be an arbitrary infinite subset of $[\supp f]$, once we
  established that $E :=\Exp^*\{\mf_g^i f \colon i \ge 0\}$ is algebraically
  independent over $\ker \Delta_J$ then its algebraic independence over
  $\cS$ follows from Lemma~\ref{l:cof2S}. Since $g$ is eventually injective,
  there exists $n_0 \in \Nn$ such that $g$ is injective and non-vanishing on
  $\{n \in \Nn \colon n \ge n_0\}$. We choose an infinite sequence from $J$
  inductively as follows: pick $p_1 \in J$ larger than $n_0$ and $p_{j+1}
  \in J$ such that
  \begin{align*}
    p_{j+1} > v(\partial_{p_j} f)p_j\qquad (j \ge 1).
  \end{align*}
  Then $v(\partial_{p_j}f)p_j$ $(j \ge 1)$ form a strictly increasing sequence
  and so the $g(v(\partial_{p_j} f)p_j)$ are nonzero and distinct. Note that
  every finite subset of $E$ is contained in $\Exp^*\{\mf_g^i f \colon k \le i
  \le k+n-1\}$ for some $k \ge 0, n \ge 1$. According to Theorem~\ref{th:[supp
  f]>n}, the latter set is algebraically independent over $\ker\{\partial_{p_j}
  \colon 1 \le j \le n\} \supseteq \ker\Delta_J$. So we conclude that $E$
  is algebraically independent over $\ker \Delta_J$. In addition, if $g$
  is nowhere vanishing, then Theorem~\ref{th:[supp f]>n} and hence the whole
  argument goes through for $E:=\Exp^*\{\mf_g^i f \colon i \in \Zz\}$.
\end{proof}
Rather curiously, for a completely additive function $g$ to be injective
means the set of complex numbers $g(\Pp)$ is $\Qq$-linearly independent;
and for a completely multiplicative $g$ to be injective means $g(\Pp)$ is
multiplicatively independent. In any case, even if one requires $\mf_g$ in
Theorem~\ref{th:m_g-trans} to be a derivation or an automorphism of $\cA$,
there are still plenty arithmetic functions that satisfy the requirements.

\begin{exa} \label{ex:zeta-hypertrans} 
  By taking the function $g$ in Theorem~\ref{th:m_g-trans} to be the real
  logarithm, we conclude that $\vec{1}$ is $\partial_L$-transcendental
  (better known as {\bf hyper-transcendental}) over $\cS$. In particular,
  that means the Riemann zeta function $\zeta(s)$ is hyper-transcendental
  over $\Cc$. Lemma~3.1 in~\cite{aids} states that the identity function (of
  a complex variable $s$) is transcendental over the ring of complex functions
  (in $s$) defined by Dirichlet series which have a proper right half-plane of
  convergence. Thus we conclude that $\zeta(s)$ is hyper-transcendental over
  $\Cc(s)$. We refer the reader to~\cite{fibo} for some historical remarks
  of this result which is usually attributed to Hilbert~\cite{hilbert}
  in the literature.
\end{exa}

\begin{exa} 
  \label{ex:carlitz} For $k \in \Zz$, let $\mathbf{I}_k$ be the arithmetic
  function $n \mapsto n^k$. In~\cite{car}, Carlitz showed that $\mathbf{I}_k$
  ($k \ge 0$) are algebraically independent over $\Cc$. Shapiro and Sparer
  generalized this result to the algebraic independence of $\mathbf{I}_k\ (k \in
  \Zz)$ over the kernel of any infinite set of basic derivations (and hence over
  $\cS$)~\cite[Theorem~3.2]{aids}. By taking $g = \vec{I}$ the identity map of
  $\Nn$ and $f = \vec{1}$ in Theorem~\ref{th:m_g-trans}, we conclude more
  generally that $\Log \mathbf{I}_k, \mathbf{I}_k$ ($k \in \Zz$) are
  algebraically independent over the kernel of any infinite set of basic
  derivations (and hence over $\cS$).
\end{exa}

By fixing $f$ to $\vec{1}$, one can view Theorem~\ref{th:m_g-trans} as a
result about algebraic independence of, $g^{\gen{k}} := \mf_g^k(\vec{1})$
$(k \ge 0)$, the powers of $g$ with respect to the pointwise product. In
fact, since $[\supp \vec{1}] = \Pp$ and $v(\partial_p \vec{1}) =1$ for each
$p$, so an assumption weaker than eventual injectivity of $g$ is enough to
guarantee algebraic independence. More precisely, we have:
\begin{cor}
  \label{c:g^<k>} $\Exp^*\{g^{\gen{i}} \colon i \ge 0\}$ is algebraically
  independent over $\Cc$ if $g(\Pp)$ is infinite and is algebraically
  independent over $\cS$ if $g(I)$ is infinite for every infinite set
  of primes $I$. Moreover, the same is true with $i$ ranging through the
  integers if $g$ is nowhere vanishing.
\end{cor}
Corollary~\ref{c:g^<k>}, in particular, implies if $g(\Pp)$ is infinite
then $g$ does not satisfy, in the algebra $(\cA, +, \cdot)$, any nontrivial
polynomial equation over $\Cc$. We will have another discussion about this
kind of independence in Section~\ref{sec:remarks}.

Let $(U_n)$ be a linear integral recurrence of order two, by that we mean
$(U_n)$ is a sequence of integers satisfying
\begin{align*}
  U_{n+2} = PU_{n+1} - QU_{n} \qquad (n \ge 1)
\end{align*}
for some $P,Q \in \Zz$ with $Q \neq 0$. Suppose $\rho$ is a ratio (the other
being $1/\rho$) of the two roots of the characteristic polynomial $z^2 -Pz + Q$.
Morgan Ward showed in~\cite[Theorem~1]{ward} that the set $\{U_n \colon n \ge
1\}$ has infinitely many prime divisors if either (1) $\rho$ is not a root of
unity, or (2) $\rho = 1$. In the first case, the recurrence $(U_n)$ is called
{\bf non-degenerate}. Thus, if $U \subseteq \Nn$ is the set of terms of a
non-degenerate second order linear integral recurrence, then $\vec{1}_U \notin
\cS$. By Theorem~\ref{th:m_g-trans}, we conclude that $\vec{1}_U$ is
$\mf_g$-transcendental over $\cS$ for any $g$ that is eventually injective.

\begin{exa} \label{ex:Fibo}
The linear recurrence defining the Fibonacci numbers: $F_1 = 1, F_2 =1$
and $F_{n+2} = F_{n+1} + F_n$ is second order and non-degenerate. Thus
$\vec{1}_F$, the indicator of function of the Fibonacci numbers, is
hyper-transcendental over $\Cc$. By an argument similar to the one given
in Example~\ref{ex:zeta-hypertrans}, we conclude that the Fibonacci zeta
function $\zeta_F(s)$ is hyper-transcendental over $\Cc(s)$.
\end{exa}

Our next result generalizes both~\cite[Theorem~3.3]{aids}
and~\cite[Theorem~3]{obs} by relaxing the assumption that $\supp f$ contains
infinitely many primes to that $\supp f$ is not finitely generated. The proof
below is a mixture of the those given in~\cite{aids} and~\cite{obs}.  Therefore,
our sole contribution here is the realization that these proofs remain valid in
a more general setting. We also hope our use of the lexicographic ordering on
the index set can clarify the presentation. In the following, $T^{\alpha}$
($\alpha \in \Cc$) stands for the operator $\mf_g$ where $g$ is the function $n
\mapsto n^{\alpha}$.

\begin{thm}
  \label{th:diff-diff} For any $f \in \cA \setminus \cS$ and any sequence
  $(\alpha_i)_{i \ge 0}$ of complex numbers with distinct real parts, the
  set of arithmetic functions
  \begin{align*}
    \Exp^*\{T^{\alpha_i} \partial_L^j f \colon i,j \ge 0\}
  \end{align*} 
  is algebraically independent over the kernel of any infinite subset of
  $\Delta_f$ and consequently over $\cS$.
\end{thm} 
\begin{proof}
  Since $f \in \cA\setminus \cS$, $[\supp f]$ is infinite and so are its
  co-finite subsets. So by Lemma~\ref{l:cof2S}, we only need to show is that
  for any $k,m \ge 0$, the set of functions
  \begin{align*}
    \Exp^*\{f_{ij} \colon 0 \le i \le k, 0 \le j \le m \},
  \end{align*}
 where $f_{ij} := T^{\alpha_i}\partial_L^j f$, is algebraically independent
 over the kernel of any infinite subset of $\Delta_f$. Let
 \begin{align*}
   L = \{(a,b) \colon 0 \le a \le k, 0 \le b \le m\}
 \end{align*} 
 be the index set ordered lexicographically. If no confusion arise, we follow
 the convention of indexing matrix entries by writing the index $(a,b)$ as
 $ab$. 
 
 Given $J$ an infinite subset of $[\supp f]$, we are going to choose a
 sequence of primes $(p_{uv} \colon (u,v) \in L)$ from $J$. Let $m_{uv}$
 be the order of $\partial_{p_{uv}} f$. By applying Lemma~\ref{l:ord_inv}
 twice, we conclude $m_{uv} = v(\partial_{p_{uv}} f_{ij})$ for any $(i,j)
 \in L$. We claim that the determinant of the $|L| \times |L|$ matrix,
 \begin{align*}
   \left( \partial_{p_{uv}}f_{ij}(m_{uv}) \right) &= \left(\prod_{(u,v) \in
   L} \partial_{p_{uv}}f(m_{uv})\right)\left( (m_{uv}p_{uv})^{\alpha_i}(\log
   (m_{uv}p_{uv}))^j\right)
 \end{align*} 
 is non-zero if we impose suitable requirements on the sequence
 $(p_{uv})$. Once this is achieved, it then follows from
 Theorem~\ref{th:m_j<=ord} that the set of arithmetic functions
 $\Exp^*\{f_{ij} \colon (i,j) \in L \}$ is algebraically independent over
 $\ker\{\partial_{p_{uv}} \colon (u,v) \in L\} \supseteq \ker \Delta_J$.

 Since $\partial_{p_{uv}} f(m_{uv}) \neq 0$ for each $(u,v) \in L$, it
 suffices to make the determinant of the matrix 
 \begin{align*}
   P := \left( (m_{uv}p_{uv})^{\alpha_i}(\log(m_{uv}p_{uv}))^j \right)
 \end{align*}
 non-zero. To achieve that, first note that the entries of $P$ are all
 nonzero and hence each term in the expansion of $\det P$ is nonzero. By
 re-arranging the $\alpha_i$ ($1 \le i \le n$), if necessary, we can assume
 their real parts form a strictly increasing sequence. Let $t_{\diag}$
 denote the product of the diagonal entries of $P$, i.e.
 \begin{align*}
   t_{\diag} = \prod_{(u,v) \in L}
   (m_{uv}p_{uv})^{\alpha_u}(\log(m_{uv}p_{uv}))^{v}.
 \end{align*} 
 The key observation is that the ratio $t/t_{\diag}$, where $t$ is any other
 term in the expansion of $\det P$, has the form
 \begin{align*}
   \prod_{(u,v) \in L}
   (m_{uv}p_{uv})^{\gamma(u,v)}(\log(m_{uv}p_{uv}))^{d(u,v)}
 \end{align*} 
 and if $(u',v') \in L$ is the largest index such that $(\gamma(u',v'),
 d(u',v'))$ is not $(0,0)$, then $(\Re(\gamma(u',v')), d(u',v')) < (0,0)$
 lexicographically.  Therefore, if we choose $(p_{uv})$ an increasing sequence
 of primes from $J$ such that each $p_{uv}$ is sufficiently large compare to
 its predecessors, for example, pick $p_{00} \ge 3$ (to ensure $\log p_{uv}
 > 1$ for all $(u,v) \in L)$) and $p_{uv}$ such that
 \begin{align*}
   \log p_{uv} > 
   |L|! \prod_{(u',v') < (u,v)} (m_{u'v'}p_{u'v'})^{|\alpha_k|+m},
 \end{align*} 
 then $|t/t_{\diag}| < (|L|!)^{-1}$. Thus for such a choice of $(p_{uv})$,
 \begin{align*}
   |\det P| \ge |t_{\diag}|\left(1 - \sum_{t \neq t_{\diag}}
   |t/t_{\diag}|\right) > 0.
 \end{align*}
\end{proof}
A couple remarks about Theorem~\ref{th:diff-diff}. First, arithmetic functions
of the form $n^{\alpha_i}(\log n)^j f(n)$ ($j \in \Zz$) were considered
in both~\cite{aids} and~\cite{obs}. This is problematic for negative $j$
since these functions are not defined at $1$ and consequently their higher
convolution powers are undefined. Second, if Theorem~\ref{th:diff-diff}
admits an ``algebraic'' proof, by that we mean a proof similar to that of
Theorem~\ref{th:m_g-trans} which does not rely on the growth rate of the
functions involved, then one may expect a generalization about operators of
the form $\mf_h^i\mf_g^j$.

\begin{cor}
  \label{c:2ndrec} Suppose $U$ is a set of natural numbers with an infinite
  set of prime divisors, then $\zeta_U(s)$ does not satisfy any nontrivial
  algebraic differential difference equation over $\Cc(s)$.
\end{cor}
\begin{proof}
  Since $\vec{1}_U \notin \cS$, Theorem~\ref{th:diff-diff} implies the set
  of arithmetic functions $\{T^{\alpha_i}\partial_L^j\vec{1}_U \colon i,j
  \ge 0\}$ is algebraically independent over $\Cc$ for any complex sequence
  $(\alpha_i)$ with distinct real parts. Since $(-1)^jT^{\alpha_i}\partial_L^j
  \vec{1}_U$ corresponds to $\zeta_U^{(j)}(s - \alpha_i)$ under the isomorphism
  in~\eqref{eq:isods}, the corollary is true over $\Cc$. Finally, by Lemma~3.1
  of~\cite{aids}, it is true for $\Cc(s)$ since the formal Dirichlet series
  involved are convergent.
\end{proof}
\begin{exa} 
  \label{ex:ost} Corollary~\ref{c:2ndrec} implies a classical result of
  Ostrowski \cite{ost}: $\zeta(s)$ does not satisfy any nontrivial algebraic
  differential difference equation over $\Cc(s)$. That means there is no
  non-zero polynomial $F(s, z_1,\ldots, z_k)$ over $\Cc$ such that the function
\begin{align*}
  F(s, \zeta^{(m_1)}(s-r_1), \ldots, \zeta^{(m_k)}(s-r_k)),
\end{align*} 
where $(m_i,r_i)$ are distinct pairs of integers and $m_i \ge 0$ for all $1
\le i \le k$, vanishes identically on its domain.
\end{exa}

\begin{exa}
  \label{ex:fibo} Recall that if $U \subseteq \Nn$ is the set of terms
  of a non-degenerate second order linear recurrences, then $\vec{1}_U
  \notin \cS$. Thus, Corollary~\ref{c:2ndrec} also implies the Fibonacci
  zeta function $\zeta_F(s)$ does not satisfy any nontrivial algebraic
  differential difference equation over $\Cc(s)$. Since it is not known
  whether the Fibonacci sequence contains infinitely many primes, this
  statement cannot be deduced, at least for now, from either Theorem~3.3
  of~\cite{aids} or Theorem~3 of~\cite{obs}.

  Many sequences of natural numbers, well-known to number theorists, are in fact
  non-degenerate second order integral linear recurrences (see~\cite{koshy} for
  a reference): Lucas sequence, Pell sequence and Pell-Lucas sequence, to name a
  few. Thus their zeta functions do not satisfy any non-trivial algebraic
  difference-differential equations over $\Cc(s)$. More generally, one can
  replace ``algebraic'' by ``holomorphic'' in the previous statement, if one
  invokes an analytic result of Reich~\cite[Satz~1]{reich} instead of
  Theorem~\ref{th:diff-diff}. This is the way in which
  Steuding~\cite[Theorem~1]{fibo} and Komatsu~\cite[Corollary~1]{komatsu}
  obtained the corresponding results for the Riemann zeta function and the Lucas
  zeta function, respectively.  In~\cite{fibo}, Steuding made no reference to
  Ward's paper~\cite{ward} but did mention that his argument obviously can be
  extended to other Dirichlet series that built from linear recurrence.
\end{exa}

\section{Remarks} \label{sec:remarks} 

We conclude with a few observations that we made along the way of studying
arithmetic functions. The first one is about derivations of $\cA$. As noted,
Theorem~\ref{th:gJac} will be an unconditional generalization of
Shapiro-Sparer's Jacobian criterion if every derivation of $\cA$ is continuous.
Unfortunately, we can neither prove that every derivation of $\cA$ is continuous
nor produce one that is not. There is indeed a construction given at the end of
Section~4 in~\cite[p.309--312]{conv} which produces nonzero derivations of $\cF$
that vanish on the $e_n$ ($n \in \Nn$) and hence $\cT$. Since $\cF$ is the field
of fractions of $\cA$, any such derivation must also be nonzero on $\cA$ but
then it cannot be continuous since $\cA$ is the closure of $\cT$ in $\cF$.
However, it is unclear to us that any derivation constructed this way actually
restricts to a map from $\cA$ to itself. Here we would like to offer a similar
but hopefully simpler way of constructing derivations $\cF$ that do not preserve
null sequences of $\cA$: start with a null sequence in $\cA$ that is
algebraically independent over $\Cc$, for example $(e_p)_{p \in \Pp}$.  Extend
it to a transcendence base $B$ of $\cF$ over $\Cc$. Then $db (b \in B)$ form a
$\cF$-basis of $\Omega_{\cF/\Cc}$~\cite[Theorem~26.5]{mat}. The derivation $D$
of $\cF$ obtained by composing $d$ with the $\Cc$-linear map determined by $db
\mapsto 1$ ($b \in B$) maps each $e_p$ to $1$ and hence cannot be a continuous
derivation of $\cA$ if it does restrict to a map from $\cA$ to itself. The flip
side of the coin is that every derivation of $\cA$ is continuous. This will be
true if the topology determined by the norm $\norm{\cdot}$ is equivalent to the
$\cI$-adic topology of some ideal $\cI$ of $\cA$. This is because for any $n \ge
1$, and $f \in \cI^n$, the derivative $f$ with respect to any derivation of
$\cA$, according to the Leibniz rule, is in $\cI^{n-1}$ and so any derivation of
$\cA$ is continuous with respect to the topology determined by any ideal of
$\cA$.  We should point out, however, in the case when $\cI$ is the unique
maximal ideal $\cA_0$ these two topologies are inequivalent. For example, none
of the term in the null sequence $(e_p)$ is even in $\cA_0^2$ because members of
$\cA_0^2$ vanish on every prime.

Our second observation is about linear independence of arithmetic functions over
$\Cc$. It was proved~\cite[Theorem~3.2--3.4]{imaf2} that arithmetic functions
$f_1,\ldots, f_n$ are linearly dependent over $\Cc$ if and only if their
Wronskian with respect to the log-derivation, i.e.
\begin{align*}
  \wL{f_1}{f_n}{n-1},
\end{align*}
vanishes identically. We claim that the same is true, more generally, for
elements of $\cF$ and offer a softer proof in the sense that no formula for the
values of Wronskian is needed. We take advantage of a standard result of
differential fields~\cite[Theorem~6.3.4]{afg} which asserts that elements of a
differential field $(F,D)$ are linearly dependent over the field of constants if
and only if their Wronskian with respect to $D$ (or $D$-Wronskian, in short) is
zero. Thus by taking the differential field to be $(\cF, \partial_L)$, all we
need to show is that the kernel of the log-derivation in $\cF$ is still $\Cc$.
Before proving that statement, it is probably worth pointing out that in general
$\ker_{\cF} D$ needs not be the fraction field of $\ker D$ in $\cF$: recall that
$\Omega$ counts the total number of prime factors with multiplicity of its
argument. One checks readily that $\ker \dO = \Cc$ and $\dO e_p = e_p$ for each
prime $p$. Thus for distinct primes $p$ and $q$, $e_p/e_q \in \ker_{\cF} \dO
\setminus \Cc$.
\begin{prop}
  $\ker_{\cF} \partial_L = \Cc$.
  \label{p:kernel}
\end{prop}
\begin{proof}
  First, $\ker_{\cF} \partial_L \supseteq \ker \partial_L = \Cc$. To establish
  the reverse inclusion, take any $f,g \in \cA\setminus\{0\}$ such that
  $\partial_L (f/g) = 0$ then
  \begin{align}
    \label{eq:quo} \partial_L f * g = f * \partial_L g.
  \end{align}
  If $g$ is invertible in $\cA$, $f/g \in \cA \cap \ker_{\cF}\partial_L =
  \Cc$. So let us assume $g$ is not invertible; that is $g(1) =0$, it then
  follows that $\norm{\partial_L g} = \norm{g} (> 0)$. Now by taking norm on
  both sides of~\eqref{eq:quo}, we see that $\norm{\partial_L f}=\norm{f}
  (> 0)$. Thus, by evaluating both sides of~\eqref{eq:quo} at $v(f)v(g)$,
  we conclude that $\log(v(f)) = \log(v(g))$ and hence $v(f) = v(g)$. Let $k$
  be this common value and $h$ be $f-\alpha g$ where $\alpha= f(k)/g(k)$. Then
  the order of $h$ is strictly greater than $k$ and $h/g = f/g- \alpha \in
  \ker_{\cF} \partial_L$. So unless $h=0$, i.e. $f/g = \alpha \in \Cc$,
  otherwise the same argument with $f$ replaced by $h$ will lead us to the
  contradicting conclusion that $v(h) = v(g) =k$. This completes the proof
  of the other inclusion.
\end{proof}

Viewing the linear independence result of~\cite{imaf2} as one about
differential fields frees us from focusing on the log-derivation: if the
Wronskian of $f_1,\ldots, f_n$ with respects to any derivation $D$ of $\cF$
is non-zero, then $f_1,\dots, f_n$ is linearly independent over $\ker_{\cF} D$
and hence over $\Cc$. Let us give an application. Recall that $g^{\gen{k}}$
($k \ge 0$) denotes the $k$-th power of $g$ with respect to the pointwise
product. Consider again the function $\Omega$. The value at 1 of the
$\partial_2$-Wronskian of $\vec{1}=\Omega^{\gen{0}}, \Omega^{\gen{1}},
\ldots, \Omega^{\gen{n}}$ is
\begin{align*}
  \det\left( \partial_2^j \Omega^{\gen{i}}(1) \right) &= \det\left( j!
  \Omega^{\gen{i}}(2^j) \right) =\det\left( j^i \right)\prod_{j=0}^n j!
\end{align*}
which is nonzero since the last determinant is Vandermonde. We conclude
that $\Omega^{\gen{k}}$ ($k \ge 0$) are linearly independent over
$\Cc$. Therefore, the $\partial_L$-Wronskian of $\vec{1}, \Omega^{\gen{1}},
\ldots, \Omega^{\gen{n}}$ must also be nonzero but this is harder to spot
since its value at 1 is 0. This also shows that $\Omega$ does not satisfy
any nontrivial polynomial equation over $\Cc$ in the $\Cc$-algebra $(\cA, +,
\cdot)$. Note that this fact cannot be deduced from Corollary~\ref{c:g^<k>}
since $\Omega(\Pp) = \{1\}$ is finite. Note also that this statement is
stronger than asserting $\Omega$ is transcendental over $\Cc$ in the sense of
Bellman and Shapiro~\cite{B-S}. Roughly speaking, since $(\cA,+, \cdot)$ is
not an integral domain, so the ``right'' definition for algebraic dependence
requires not just a nontrivial polynomial but an irreducible one to vanish
at the functions involved.

Our last few remarks are about Theorem~\ref{th:gJac} and Section~2
of~\cite{imaf2}. In searching for a generalization of the Jacobian criteria,
we realize that the derivations in Theorem~\ref{th:gJac} cannot be replaced
by differential operators. More precisely, consider, for each $k \in \Nn$,
the differential operator $\partial_k := \prod_{p} \partial_p^{v_p(k)}$
(here the product is composition of functions). One checks that for $f \in
\cA$ and $n \in \Nn$,
\begin{align*}
  (\partial_k f)(n) = f(kn) \prod_p\prod_{j =1}^{v_p(k)}\left( v_p(n)+j
  \right).
\end{align*}
In particular, $(\partial_k f)(1) = f(k)\prod_p (v_p(k)!)$. Thus, if we
normalize $\partial_k$ to
\begin{align*}
  \hat{\partial}_k = \left( \prod_p (v_p(k)!) \right)^{-1}\partial_k,
\end{align*}
then we will have $\varepsilon_1 \circ \hat{\partial}_k = \varepsilon_k$. To
see that Theorem~\ref{th:gJac} fails if we replace the derivations
by differential operators, take $f_1$ to be $\vec{1}_2$ and $f_2 =
f_1*f_1$. Note that
  \begin{align*}
    f_2(n) = 
    \begin{cases}
      k+1 & \text{if}\ n = 2^k\ \text{for some $k \ge 0$;} \\ 
      0   & \text{otherwise}.
    \end{cases}
  \end{align*}
  Certainly, $f_1$ and $f_2$ are algebraically dependent over $\Cc$ but
  \begin{align*}
    \det
  \begin{pmatrix}
    \hat{\partial}_2 f_1 & \hat{\partial}_4 f_1 \\
    \hat{\partial}_2 f_2 & \hat{\partial}_4 f_2
  \end{pmatrix}
  (1) = \det
  \begin{pmatrix}
    f_1(2) & f_1(4) \\
    f_2(2) & f_2(4)
  \end{pmatrix} 
  =
  \det
  \begin{pmatrix}
    1 & 1 \\
    2 & 3
  \end{pmatrix}
  = 1.
\end{align*}
Incidentally, this shows that Theorem~2.2 of~\cite{imaf2} is not true.
Moreover,
\begin{align*}
  \partial_L f_1(n) =
  \begin{cases}
    k\log(2) & \text{if}\ n = 2^k\ \text{for some $k \ge 0$;} \\ 0
    &\text{otherwise}.
  \end{cases}
\end{align*}
Thus $f_1$ satisfies the following differential algebraic equation over $\Cc$:
\begin{align*}
  \partial_L X = \log(2)(X^2-X).
\end{align*}
This falsifies Corollary~2.3--2.5 of~\cite{imaf2}. In particular, it is
not true that a Dirichlet series which is not a Dirichlet polynomial is
hyper-transcendental over $\Cc$. Corollary~2.6 and~2.7 of~\cite{imaf2} are
also problematic. Again the pair $(f_1,f_2)$ furnishes a counterexample to
Corollary~2.7 of~\cite{imaf2} which asserts that arithmetic functions with
infinite supports are algebraically independent over $\Cc$. Since $f_1,f_2$
are algebraically dependent over $\Cc$, so are the arithmetic functions
$g_1,g_2$ defined by
\begin{align*}
  \begin{pmatrix}
    g_1 \\ g_2
  \end{pmatrix}
  =
  \begin{pmatrix}
    f_1(2) & f_1(4) \\
    f_2(2) & f_2(4)
  \end{pmatrix}^{-1}
  \begin{pmatrix}
    f_1 \\ f_2
  \end{pmatrix}
  =
  \begin{pmatrix}
    \phantom{-}3 & -1 \\
    -2 & \phantom{-}1
  \end{pmatrix}
  \begin{pmatrix}
    f_1 \\ f_2
  \end{pmatrix}.
\end{align*}
Consequently, $h_1:=g_1-2$ and $h_2:=g_2+1$ are also algebraically dependent
over $\Cc$. Since the first four values of $h_1$ and $h_2$ are $(0,1,0,0)$
and $(0,0,0,1)$, respectively, the pair $(h_1,h_2)$ provides a counterexample
to Corollary~2.6 of~\cite{imaf2}.

{\bf Acknowledgements.} First of all, I would like to thank George Jennings
for a discussion that led me to realize my mistake in generalizing
Theorem~\ref{th:gJac}. I would also like to thank Alexandru Buium for
discussing with me the existence of non-continuous derivations of $\cA$. Our
conversations saved me from going down several dead-end paths. I thank
the referee for helpful suggestions on improving the presentation of
several results in this article, especially Theorem~\ref{th:Log_power},
Theorem~\ref{th:diff-diff} and Proposition~\ref{p:kernel}.

\section{Addendum} 

In this addendum, we argue an unconditional generalization of the Jacobian
criteria for arithmetic functions. The point being overlooked in the original
article is that even though there may be derivations of $\cA$ that are not
continuous with respect to the norm topology, as we remarked earlier, every
derivation of $\cA$ is continuous with respect to the $\mf$-adic topology. So
our original proofs will go through once we can demonstrate that convergence
for power series behave the same way in these topologies. Before we do so,
it is perhaps worth mentioning that for $\cA$ the $\mf$-adic topology is
strictly finer than the norm topology. This can be seen from the fact that
elements of $\mf^n$ have order at least $2^n$. Thus any sequence of elements
of $\cA$ that converges in the $\mf$-adic topology also converges in the norm
topology. The converse is not true as we pointed out (consider the sequence
$(e_p)_{p \in \Pp}$)in the previous section.

The situation is different for power series. Recall that for any $g \in \mf$
and any formal power series $\sum_{k=0}^\infty \alpha_kX^k$ over $\Cc$, the
sequence $\sum_{k=0}^m \alpha_k g^k$ $(m \in \Nn)$ converges in the norm
topology to a unique element of $\cA$. We denote it by $\sum_{k=0}^\infty
\alpha_k g^k$. In particular, for any $N \in \Nn$, $\sum_{k=0}^\infty
\alpha_{k+N}g^k$ is also a well-defined element of $\cA$. Since the convolution
production is continuous with respect to the norm topology,
\begin{align*}
  \sum_{k=0}^\infty \alpha_k g^k - \sum_{k=0}^{N-1} \alpha_kg^k
  = \sum_{k=N}^\infty \alpha_k g^k = g^{N}*\left(\sum_{k=0}^\infty
  \alpha_{k+N}g^k \right)
\end{align*}
is an element of $\mf^N$. Thus, we conclude that
\begin{prop}
Let $g \in \mf$ and $\sum \alpha_k X^k$ a formal power series over $\Cc$,
the sequence $\sum_{k=0}^m \alpha_k g^k$ $(m \in \Nn)$ converges to
$\sum_{k=0}^\infty \alpha_k g^k$ $\mf$-adically.  \label{p:m-conv}
\end{prop}
This proposition together with the fact that $\cA$ is $\mf$-adically separated,
i.e. $\bigcap_{n \in \Nn} \mf^n=\{0\}$, imply that Lemma~\ref{l:contd} holds
for any derivation of $\cA$ (same proof but with the norm topology replaced
by the $\mf$-adic topology). From this it follows that the assumption on
continuity with respect to the norm topology of the derivations involved
in all our results, including Ax's Theorem (Theorem~\ref{th:sa}) for $\cA$
and the Jacobian criterion for $\cA$ (Theorem~\ref{th:gJac}), can be removed.

\end{document}